\documentclass[12pt]{article}
\usepackage{setspace}
\usepackage{amsmath,amssymb}
\usepackage{mathtools}
\usepackage{mathrsfs}
\usepackage{geometry}
\usepackage{enumitem}
\usepackage{environ}
\usepackage{setspace}
\usepackage{titlesec}
\usepackage[all,cmtip]{xy}
\usepackage{tikz}
\usetikzlibrary{matrix,arrows}

\usepackage[colorlinks=true,citecolor=blue]{hyperref}
\usepackage[amsthm,thmmarks]{ntheorem}

\NewEnviron{prf}[1][]{\begin{proof}[\bf #1Proof]\BODY\end{proof}}{}
\NewEnviron{slt}[1][]{\begin{proof}[\bf #1	解]\BODY\end{proof}}{}
\newtheorem{definition}{Definition}[section]
\newtheorem{theorem}[definition]{Theorem}
\newtheorem{lemma}[definition]{Lemma}
\newtheorem{proposition}[definition]{Proposition}
\newtheorem{remark}[definition]{Remark}

\newtheorem{notation}[definition]{Notation}

\newtheorem*{DML}{Dynamical Mordell--Lang Conjecture (DML Conjecture)}

\newcommand\address[1]{#1}
\newcommand\email[1]{\emph{Email address}: #1}
  
\setlist[enumerate,1]{label=(\roman*)}
\setlist[enumerate,2]{label=(\alph*)}

\title{\textbf{On the DML(1) property for regular endomorphisms of affine spaces: the $\mathbb{G}_m$-case}}
\author{She Yang\quad and\quad Aoyang Zheng}
\date{}

\begin{document}
\begin{spacing}{1.25}

\maketitle

\begin{abstract}
Let $f$ be a regular endomorphism of $\mathbb{A}_{\mathbb{C}}^N$ and let $C\subseteq\mathbb{A}_{\mathbb{C}}^N$ be an irreducible curve. Suppose $C$ has an infinite intersection with the $f$-orbit of a point $x\in\mathbb{A}^N(\mathbb{C})$. Then the normalization of $C$ is isomorphic to either $\mathbb{A}^1$ or $\mathbb{G}_m$. We prove that $C$ is $f$-periodic in the latter case, as expected by the dynamical Mordell--Lang conjecture.
\end{abstract}

\section{Introduction}

In this article, the base field is $\mathbb{C}$ unless otherwise stated. As a matter of convention, every variety is assumed to be integral, but the closed subvarieties can be reducible. Let $X$ be a variety and let $f$ be an endomorphism of $X$. For a point $x\in X$, the orbit $\mathcal{O}_{f}(x)$ is the set $\{f^{n}(x)\mid n\in\mathbb{N}\}$. We denote $\mathbb{N}=\mathbb{Z}_{+}\cup\{0\}$.

The dynamical Mordell--Lang conjecture is one of the core problems in the field of arithmetic dynamics. It was proposed by Ghioca and Tucker in \cite{GT09} and can be stated as follows.

\begin{DML}
Let $f$ be an endomorphism of a quasi-projective variety $X$. Let $V$ be a positive-dimensional irreducible closed subvariety of $X$ and let $x\in X(\mathbb{C})$ be a point. Suppose $\mathcal{O}_f(x)\cap V$ is Zariski dense in $V$. Then $V$ is $f$-periodic, i.e. we have $f^n(V)\subseteq V$ for some positive integer $n$.
\end{DML}

Here, we adapt the ``geometric formulation" of the DML conjecture, while in the literature people tend to state the DML conjecture in its ``arithmetic form" which asserts that the return set of an orbit into a closed subvariety must be a finite union of arithmetic progressions. It is not hard to prove the equivalence of these two formulations. For example, see \cite[Subsection 3.1.3]{BGT16}.

There is an extensive literature on various cases of the DML conjecture. Two significant cases are as follows.

\begin{enumerate}
\item
If $f$ is an \'etale endomorphism of $X$, then the DML conjecture holds. See \cite{Bel06} and \cite[Theorem 1.3]{BGT10}.
\item
If $X = \mathbb{A}^2$, then the DML conjecture holds. See \cite{Xie17} and \cite[Theorem 4]{Xie}.
\end{enumerate}

One can consult \cite{BGT16,Xie} and the references therein for further known results.

In this paper, we mainly study the DML conjecture for curves. So we recall the following notion introduced in \cite[Definition 3]{Xie}.

\begin{definition}\label{DML1}
For a quasi-projective variety $X$ and an endomorphism $f$ of $X$, we say $(X,f)$ satisfies the \emph{DML(1) property} if for every irreducible closed subcurve $C\subseteq X$ and every point $x\in X(\mathbb{C})$, we have $C$ is $f$-periodic once $\mathcal{O}_f(x)\cap C$ is infinite.
\end{definition}

Here ``1" stands for the dimension of the closed subvariety, and we also adapt the geometric formulation.

We study the DML(1) property of regular endomorphisms of affine spaces.

\begin{definition}\label{regendo}
A \emph{regular endomorphism} of $\mathbb{A}^N$ is a dominant endomorphism $f$ that can be extended to an endomorphism of $\mathbb{P}^N$. Here, we fix the standard coordinate system $[x_0,\dots,x_N]$ on $\mathbb{P}^N$ and naturally regard $\mathbb{A}^N$ as the open subset $\{x_0\neq0\}$ in it. Therefore, the regular endomorphism $f$ can be written as
$$
f(x_1,\dots,x_N)=(f_1(x_1,\dots,x_N)+g_1(x_1,\dots,x_N),\dots,f_N(x_1,\dots,x_N)+g_N(x_1,\dots,x_N)),
$$
in which $f_1(x_1,\dots,x_N),\dots,f_N(x_1,\dots,x_N)$ are homogeneous polynomials of degree $d\geq1$ that have no nonzero common zeros and $g_1(x_1,\dots,x_N),\dots,g_N(x_1,\dots,x_N)$ are polynomials of degree $<d$.
\end{definition}

The concept of regular endomorphisms was introduced in \cite{BJ00} and had been studied in several previous works. For example, please see \cite[Subsections 1.4--1.5]{Xie24}, \cite{DFR}, \cite[Subsection 1.2]{Zhong}, and \cite[Subsection 1.2]{JXZ}. The study of the DML conjecture for regular endomorphisms goes back to Xie's groundbreaking work \cite{Xie17}. See also \cite{YZ}.

Now we state our main result.

\begin{theorem}\label{mainthm}
Let $f$ be a regular endomorphism of $\mathbb{A}^N$. Let $C\subseteq\mathbb{A}^N$ be an irreducible closed subcurve and let $x\in\mathbb{A}^N(\mathbb{C})$ be a point. Suppose $\mathcal{O}_f(x)\cap C$ is infinite. Then the normalization of $C$ is isomorphic to either $\mathbb{A}^1$ or $\mathbb{G}_m$. Moreover, $C$ is $f$-periodic in the latter case.
\end{theorem}

Therefore, in order to prove the DML(1) property for regular endomorphisms, our result solves the $\mathbb{G}_m$-case and leaves the $\mathbb{A}^1$-case open. Through the approach in \cite{Xie17} for endomorphisms of $\mathbb{A}^2$, the $\mathbb{G}_m$-case is technically more complicated than the $\mathbb{A}^1$-case. However, currently our approach cannot deal with the $\mathbb{A}^1$-case.

The first part of Theorem \ref{mainthm} is an easy consequence of Siegel's theorem on integral points Theorem \ref{siegel}. To prove the ``moreover" part, our first step is the same as \cite[p. 6, Step 1]{YZ}, which is learned form \cite[Section 8]{Xie17}. This step reduces the problem to the following proposition.

Throughout this article, an \emph{arithmetic function field} means a finitely generated field extension over $\mathbb{Q}$.

\begin{proposition}\label{mainprop}
Let $K\subseteq\mathbb{C}$ be an arithmetic function field. Let $f$ be a regular endomorphism of $\mathbb{A}^N$ with $K$-coefficients. Suppose $\deg(f)>1$. Let $(C_{-n})_{n\geq0}$ be a sequence of irreducible closed subcurves of $\mathbb{A}_{K}^N$ which satisfies $f(C_{-n})=C_{-n+1}$ for every $n\geq1$. Suppose the normalization of each $C_{-n}$ is isomorphic to $\mathbb{G}_{m,K}$. Then there are only finitely many different curves in $(C_{-n})_{n\geq0}$. Hence all of them are $f$-periodic.
\end{proposition}

\begin{remark}
It is interesting that the conclusion no longer holds if $K$ is a finitely generated field of positive characteristic (i.e. function fields over $\mathbb{F}_p$). We leave the construction of counterexamples to the interested readers.
\end{remark}

The structure of this article is as follows. We make some preparations in Section \ref{Sec2}, and then prove the main results in Section \ref{Sec3}.

\textbf{Statement on A.I. use.} The authors notice that Theorem \ref{mainthm} can be reduced to Proposition \ref{mainprop}, and it is likely to prove Proposition \ref{mainprop} in such a $\mathbb{G}_m$-case. This is because $\mathbb{G}_m$ is too rigid --- its self-maps are monomials. The splendid, essentially elementary proof of Proposition \ref{mainprop} is learned from ChatGPT 6 Astra. We clear up its proof, rearrange the material, and type the manuscript.

\section{Preparations}\label{Sec2}

In subsection \ref{sec2.1}, we recall the Arakelov theory for arithmetic function fields. Then in subsection \ref{expand}, we use this language of Arakelov geometry to show that regular endomorphisms are ``expanding". In subsection \ref{sec2.2}, we prove a linear algebra lemma that will be used later.

\subsection{Arakelov theory for arithmetic function fields}\label{sec2.1}

In the following, we closely follow \cite[Subsections 3A--3C]{Voj21}.

Let $K$ be an arithmetic function field. Let $d=\mathrm{tr.deg}_{\mathbb{Q}}(K)$. Let $B$ be a normal arithmetic variety, whose function field is isomorphic to $K$. Let $\mathcal{M}$ be a big and nef smoothly metrized line sheaf on $B$. Then we get a big polarization $M=(B;\mathcal{M})$. This is an abbreviation of $(B;\mathcal{M},\dots,\mathcal{M})$ in which there are $d$ copies of $\mathcal{M}$. These data are fixed once and for all.

The polarization $M$ gives a set of absolute values $M_{K}=M_{K}^{\infty}\sqcup M_{K}^{0}$ of $K$ in the following way.

Let $B^{(1)}$ be the set of prime Weil divisors on $B$. The set $M_{K}^{0}$ of non-archimedean absolute values will have a bijection with $B^{(1)}$. For every $Y\in B^{(1)}$, let $h_M(Y)=c_1(\mathcal{M}|_{Y})^d$ be the arithmetic intersection number. As $h_M(Y)\geq0$, the formula $\|x\|_Y=\mathrm{e}^{-h_M(Y)\mathrm{ord}_Y(x)}$ for $x\in K^{\times}$ defines a non-archimedean absolute value of $K$. This finishes the description of $M_K^{0}$.

The set $M_{K}^{\infty}$ of archimedean absolute values will have a bijection with $B(\mathbb{C})^{\mathrm{gen}}:=B(\mathbb{C})\setminus\bigcup\limits_{Y\in B^{(1)}}Y(\mathbb{C})$. For every $b\in B(\mathbb{C})^{\mathrm{gen}}$, the formula $\|x\|_b=|x(b)|$ for $x\in K$ defines an archimedean absolute value of $K$. This finishes the description of $M_K^{\infty}$.

Now we make $M_K$ into a measure space. Let $\mu_{\mathrm{fin}}$ be the counting measure on $M_K^0$, and let $\mu_{\infty}$ be the Lebesgue measure on $B(\mathbb{C})$ associated to the semipositive $(d,d)$-form $c_1(\|\cdot\|_{\mathcal{M}})^d$. Combining $\mu_{\mathrm{fin}}$ and $\mu_{\infty}$, we get a measure $\mu$ on $B(\mathbb{C})\sqcup M_K^0\supseteq M_K$. Notice that $B(\mathbb{C})\setminus M_{K}^{\infty}$ has measure zero and $M_K^{\infty}$ has positive finite measure.

In this setting, we have a \emph{product formula} which says that $\int_{M_K}\log\|x\|_vd\mu(v)=0$ for every $x\in K^{\times}$. See \cite[Section 3.2]{Mor00}.

~

Next, for every finite field extension $K'$ over $K$, we want to get the induced measure space $M_{K'}$ of absolute values on $K'$. To this end, we recall the induced big polarization $M'=(B';\mathcal{M}')$ of $K'$ \cite[Definition 3.9]{Voj21}.

Let $B'$ be the relative normalization of $B$ in $\mathrm{Spec}(K')$. Then $B'$ is also a normal arithmetic variety and the associated map $\pi\colon B'\to B$ is a finite surjective morphism of degree $[K':K]$. Let $\mathcal{M}'=\pi^*\mathcal{M}$. Then $\mathcal{M}'$ is a big and nef smoothly metrized line sheaf on $B'$, and hence we get the big polarization $M'=(B';\mathcal{M}')$. Then $M'$ induces the measure space $M_{K'}$ of absolute values on $K'$ in the sense as above.

For $v\in M_{K}$ and $w\in M_{K'}$, we can talk about ``$w$ lies over $v$" and write $w\mid v$ in the natural sense. See \cite[Definition 3.10]{Voj21}. For each $w\in M_{K'}$ lying over $v$, we introduce the indices $n_{w/v}$ following \cite[Proposition 3.11]{Voj21}.
\begin{enumerate}
\item
If $v$ is archimedean, let $n_{w/v}=1$.
\item
If $v$ corresponds to the prime Weil divisor $Y\subseteq B$ and $w$ corresponds to the prime Weil divisor $Y'\subseteq B'$, then we have $\pi(Y')=Y$. Let $e_{w/v}$ be the ramification index of the extension $\mathcal{O}_{B',Y'}/\mathcal{O}_{B,Y}$ between DVRs and let $f_{w/v}=[K(Y'):K(Y)]$. We put $n_{w/v}=e_{w/v}f_{w/v}$.
\end{enumerate}

Then the following holds \cite[Proposition 3.11]{Voj21}.
\begin{enumerate}
\item
Let $i\colon K\hookrightarrow K'$ be the field injection. For each $w\in M_{K'}$ lying over $v$ and every $x\in K$, we have $\|i(x)\|_w=\|x\|_v^{n_{w/v}}$.
\item
For every $v\in M_K$, we have $\sum\limits_{w\mid v}n_{w/v}=[K':K]$.
\end{enumerate}

~

The machinery above of inducing absolute values on finite field extensions is compatible with towers of extensions. Hence we can define heights for elements in $\overline{K}$.

\begin{definition}\label{height}
The \emph{height function} $h\colon\overline{K}\to\mathbb{R}_{\geq0}$ is defined by
$$h(x)=\frac{1}{[K(x):K]}\int_{M_{K(x)}}\log^+\|x\|_wd\mu_{K(x)}(w).$$
\end{definition}

Here and in the sequel, we denote $\log^{+}x=\max\{0,\log x\}$. Also, in the definition above, we may change the field $K(x)$ into any finite extension of it and get the same value.

Using Arakelov intersection theory, we can get the height machinery in a more general setting. See \cite[Section 3.3]{Mor00} (and also \cite[Definition 3.14]{Voj21}). Please see \cite[Proposition 3.3.2]{Mor00} for the connection between heights defined by arithmetic intersection numbers and Definition \ref{height} above.

The following theorem is the main result in Moriwaki's height machinery.

\begin{theorem}(Northcott's finiteness property; \cite[Theorem 4.3]{Mor00})\label{northcott}
For all $C>0$ and $n\in\mathbb{Z}_+$, the set $\{x\in\overline{K}\mid h(x)\leq C\text{ and }[K(x):K]\leq n\}$ is finite.
\end{theorem}

Finally, we state Siegel's theorem on integral points on curves due to Lang.

\begin{theorem}\label{siegel}
Let $R\subseteq K$ be a finitely generated subring over $\mathbb{Z}$. Let $C\subseteq\mathbb{A}_K^N$ be an irreducible closed subcurve. Suppose $C(K)\cap R^N$ is infinite. Then the normalization $\widetilde{C}$ of $C$ admits an open immersion into $\mathbb{P}_K^1$ such that the complement $(\mathbb{P}_K^1\setminus\widetilde{C})_{\mathrm{red}}$ is a zero cycle of degree less than or equal to $2$.
\end{theorem}

The reference is \cite[Corollary 4.11]{Voj21}. Since our statement is not literally the same as in the reference, we make some brief explanations. As in the proof of \cite[Corollary 4.11]{Voj21}, we denote $X_0$ as the projective closure of $C$ and let $\pi\colon X\to X_0$ be the normalization of $X_0$. Since $C$ contains infinitely rational points, so do all the curves in here and hence they are geometrically integral. The proof of loc. cit. establishes the facts that $X$ has genus 0 and the degree of the zero cycle $(X\setminus\pi^{-1}(C))_{\mathrm{red}}$ is at most 2. Combining these informations, we know $X\cong\mathbb{P}_K^1$ \cite[Proposition 53.10.4 (0C6U)]{Stacks}, and hence the result follows because $\pi^{-1}(C)$ is naturally isomorphic to the normalization of $C$.

\subsection{Regular endomorphisms are expanding}\label{expand}

In this subsection, we let $K\subseteq\mathbb{C}$ be an arithmetic function field and let $f$ be a regular endomorphism of $\mathbb{A}^N$ with $K$-coefficients. The main result of this subsection is Lemma \ref{explem}, which describes the expanding nature of $f$ at each place of $K$. Before stating the result, we need to recall the concept of ``$M_K$-constants" \cite[Definition 3.18]{Voj21}. We fix a big polarization $M=(B;\mathcal{M})$ of $K$ to run the machinery introduced in subsection \ref{sec2.1}.

\begin{definition}\label{MKconst}
An $M_K$-\emph{constant} is a measurable, $L^1$ function from $M_K$ to $\mathbb{R}_{\geq0}$ whose support has finite measure. We denote an $M_K$-constant as $(c_v)_{v\in M_K}$.
\end{definition}

\begin{remark}\label{MKconstrmk}
\begin{enumerate}
\item
For an $M_K$-constant $(c_v)_{v\in M_K}$, we have $c_v=0$ for all but finitely many $v\in M_K^0$. The measurable and $L^1$ conditions are mainly toward the behavior of $c_v$ for $v\in M_{K}^{\infty}$.
 
\item
For every $x\in K$, the function $(\log^+\|x\|_v)_{v\in M_K}$ is an $M_K$-constant.

\item
The concept of $M_K$-constant is compatible with finite field extensions. Namely, let $K'$ be a finite extension of $K$ and let $(c_v)_{v\in M_K}$ be an $M_K$-constant. For $w\in M_K'$, let $c_w=n_{w/v}c_v$ where $v\in M_K$ is the place that $w$ lies on. Then $(c_w)_{w\in M_{K'}}$ is an $M_{K'}$-constant such that $\int_{M_{K'}}c_wd\mu_{K'}(w)=[K':K]\int_{M_K}c_vd\mu_K(v)$.

\item
Our definition is slightly different from \cite[Definition 3.18]{Voj21} that we require $M_K$-constants to take nonnegative values.
\end{enumerate}
\end{remark}

In the following, we suppose the degree of the regular endomorphism $f$ is \emph{greater than $1$}.

\begin{lemma}\label{explem}
There exists an $M_K$-constant $(c_v)_{v\in M_K}$ such that the following holds.

Let $v\in M_K$. Let $E$ be an arbitrary field extension of $K$ and $u$ be an absolute value on $E$ which extends $v$. Then for every $x=(x_1,\dots,x_N)\in E^N$, we have $\|f(x)\|_u>\|x\|_u$ when $\|x\|_u>\mathrm{e}^{c_v}$. Here $\|x\|_u$ stands for the $L^{\infty}$-norm, i.e. $\|x\|_u=\max\limits_{1\leq i\leq N}\|x_i\|_u$.
\end{lemma}

We introduce the following notations to simplify the writing of proofs.

\begin{notation}\label{notation}
For a polynomial $F(x_1,\dots,x_n)=\sum\limits_{i_1,\dots,i_n\geq0}a_{i_1,\dots,i_n}x_1^{i_1}\cdots x_n^{i_n}$ with $K$-coefficients and an absolute value $v\in M_K$, we denote
$$
H_v(F)=\max\limits_{(i_1,\dots,i_n)\in\mathbb{N}^n}\|a_{i_1,\dots,i_n}\|_v\quad\text{and}\quad L_v(F)=\sum\limits_{(i_1,\dots,i_n)\in\mathbb{N}^n}\|a_{i_1,\dots,i_n}\|_v.
$$
\end{notation}

\proof[Proof of Lemma \ref{explem}]
Write 
$$
f(x_1,\dots,x_N)=(f_1(x_1,\dots,x_N)+g_1(x_1,\dots,x_N),\dots,f_N(x_1,\dots,x_N)+g_N(x_1,\dots,x_N))
$$
as in Definition \ref{regendo}. We have $f_1(x_1,\dots,x_N),\dots,f_N(x_1,\dots,x_N)$ are homogeneous polynomials of degree $d\geq2$ that have no nonzero common zeros in $\overline{K}$ and $g_1(x_1,\dots,x_N),\dots,g_N(x_1,\dots,x_N)$ are polynomials of degree $<d$. According to the homogeneous Nullstellensatz, there is an integer $q\geq d$ and homogeneous $H_{ij}\in K[x_1,\dots,x_N]$ of degree $q-d$ such that
$$
x_j^q=\sum\limits_{i=1}^{N}H_{ij}(x_1,\dots,x_N)f_i(x_1,\dots,x_N)
$$
holds for every $1\leq j\leq N$.

For $v\in M_K$, we set
$$
c_v=\left\{
\begin{array}{cc}
\log\left(1+\sum\limits_{1\leq i,j\leq N}L_v(H_{ij})\right)+\log\left(1+\sum\limits_{i=1}^{N}L_v(g_i)\right), & v\in M_{K}^{\infty} \\
\log^+\max\limits_{1\leq i,j\leq N}H_v(H_{ij})+\log^+\max\limits_{1\leq i\leq N}H_v(g_i), & v\in M_K^0
\end{array}
\right..
$$
One can check that $(c_v)_{v\in M_K}$ is an $M_K$-constant. We check that these $c_v$ satisfy the required property.

Write $x=(x_1,\dots,x_N)$. Denote $\tilde{f}(x)=(f_1(x),\dots,f_N(x))$ and $g(x)=(g_1(x),\dots,g_N(x))$.

Let $v\in M_K$ and let $x\in E^N$ satisfy $\|x\|_u>\mathrm{e}^{c_v}$. If $v\in M_K^{\infty}$, then we have
$$
\|x\|_u^q\leq\|\tilde{f}(x)\|_u\cdot\|x\|_u^{q-d}\cdot\sum\limits_{1\leq i,j\leq N}L_v(H_{ij})\quad\text{and}\quad\|g(x)\|_u\leq\|x\|_u^{d-1}\cdot\sum\limits_{i=1}^{N}L_v(g_i)
$$
and hence
\begin{align*}
\|f(x)\|_u\geq\|\tilde{f}(x)\|_u-\|g(x)\|_u
&\geq\|x\|_u^{d-1}\cdot\left(\frac{\|x\|_u}{\sum\limits_{1\leq i,j\leq N}L_v(H_{ij})}-\sum\limits_{i=1}^{N}L_v(g_i)\right) \\
&>\|x\|_u\cdot\left(\frac{\mathrm{e}^{c_v}}{\sum\limits_{1\leq i,j\leq N}L_v(H_{ij})}-\sum\limits_{i=1}^{N}L_v(g_i)\right)>\|x\|_u.
\end{align*}
If $v\in M_K^0$, then the estimates become
$$
\|x\|_u^q\leq\|\tilde{f}(x)\|_u\cdot\|x\|_u^{q-d}\cdot\max\limits_{1\leq i,j\leq N}H_v(H_{ij})\quad\text{and}\quad\|g(x)\|_u\leq\|x\|_u^{d-1}\cdot\max\limits_{1\leq i\leq N}H_v(g_i).
$$
The inequality $\|x\|_u>\mathrm{e}^{c_v}$ implies $\|\tilde{f}(x)\|_u>\|g(x)\|_u$. Thus we can also get
$$
\|f(x)\|_u=\|\tilde{f}(x)\|_u\geq\frac{\|x\|_u^d}{\max\limits_{1\leq i,j\leq N}H_v(H_{ij})}>\|x\|_u\cdot\frac{\mathrm{e}^{c_v}}{\max\limits_{1\leq i,j\leq N}H_v(H_{ij})}\geq\|x\|_u.
$$
Hence we finish the proof.
\endproof

\subsection{A linear algebra lemma}\label{sec2.2}

In this subsection, we prove a linear algebra lemma that will play a crucial role in the proof later. The main statement is Proposition \ref{linlem}.

In the following, we let $V$ be a finite-dimensional normed vector space over $\mathbb{R}$. We start with some basic definitions and properties.

\begin{definition}\label{polydef}
\begin{enumerate}
\item
A \emph{polytope} in $V$ is the convex hull of finitely many points in $V$.
\item
Let $P\subseteq V$ be a polytope. A point $v\in P$ is said to be a \emph{vertex} of $P$ if $P\setminus\{v\}$ is still convex.
\end{enumerate}
\end{definition}

\begin{lemma}\label{polylem}
Let $P\subseteq V$ be a polytope. Suppose $P$ is the convex hull of points $v_1,\dots,v_n$.
\begin{enumerate}
\item
Let $H\subseteq V$ be a linear coset. Then $P\cap H$ is also a polytope.
\item
Each vertex of $P$ lies in $\{v_1,\dots,v_n\}$, and $P$ is the convex hull of its finitely many vertices.
\end{enumerate}
\end{lemma}

\begin{proof}
\begin{enumerate}
\item
We may assume that $H$ is a hyperplane and $P\cap H$ is nonempty. Suppose $H$ is defined by the linear equation $l(v)=0$. Let $A=\{v_1,\dots,v_n\}\cap H$ and
$$
B=\{\overline{v_iv_j}\cap H\mid l(v_i)l(v_j)<0;1\leq i<j\leq n\}
$$
where $\overline{v_iv_j}$ stands for the segment joining $v_i$ and $v_j$. We claim that $P\cap H$ is the convex hull of the finite set $A\cup B$. Since $P\cap H$ is convex and contains $A\cup B$, it contains the convex hull of $A\cup B$. Now we establish the other inclusion.

Take $x\in P\cap H$ and write $x=\sum\limits_{i=1}^n\lambda_iv_i$ with each $\lambda_i\geq0$ and $\sum\limits_{i=1}^n\lambda_i=1$. We apply the following operation to turn this expression into $x=\sum\limits_{v_i\in A}\lambda_iv_i+\sum\limits_{q\in B}\varepsilon_qq$ with each $\varepsilon_q\geq0$ and $\sum\limits_{v_i\in A}\lambda_i+\sum\limits_{q\in B}\varepsilon_q=1$.

If the expression does not have the expected form, then there exist $1\leq i<j\leq n$ such that $\lambda_i,\lambda_j>0$ and $l(v_i)l(v_j)<0$. Write the intersection point $q$ of $\overline{v_iv_j}$ and $H$ as $q=tv_i+(1-t)v_j$ with $0<t<1$, and set $\varepsilon=\min\{\frac{\lambda_i}{t},\frac{\lambda_j}{1-t}\}$. Then we have
$$
\lambda_iv_i+\lambda_jv_j=\varepsilon q+(\lambda_i-\varepsilon t)v_i+(\lambda_j-\varepsilon(1-t))v_j.
$$

In this new expression, all coefficients remain nonnegative and their total is unchanged while at least one of the two coefficients of $v_i$ and $v_j$ now becomes zero. Hence after applying this operation finitely many times, we can get the desired expression and conclude that $P\cap H$ is contained in the convex hull of $A\cup B$. Therefore, $P\cap H$ is equal to this convex hull and it is also a polytope.

\item
The first part of the assertion is easy, and the second part is a baby case of the Krein--Milman theorem \cite[Theorem 3.23]{Rud91}.
\end{enumerate}
\end{proof}

\begin{definition}\label{parlinisom}
A \emph{partial linear isometry} in $V$ means a linear isometry $T\colon U\to V$ onto its image, where $U$ is a linear subspace of $V$. The \emph{rank} of this partial linear isometry $T$ is the dimension of its definition domain $U$.
\end{definition}

The partial linear isometries of $V$ form a semigroup under partial composition.

\begin{proposition}\label{linlem}
Suppose that the closed unit ball of $V$ is a polytope. Let $T_1,\dots,T_n$ be partial linear isometries of $V$. Then the semigroup generated by $T_1,\dots,T_n$ is finite.
\end{proposition}

\begin{proof}
We prove by induction on the maximum $r$ of the rank of $T_1,\dots,T_n$. The assertion is valid when $r=0$. Assume the assertion holds when $r<r_0$. We consider the case $r=r_0$.

We first show that there are only finitely many rank-$r_0$ elements in this semigroup. If an element $T_{i_k}\circ\cdots\circ T_{i_1}$ has rank $r_0$, then each $T_{i_j}$ in this expression has rank $r_0$. Moreover, the image and domain of definition of each pair of adjacent generators coincide. Therefore, both the domain of definition and the image of this element have to be the same as one of the generators. Thus in order to finish this step, we only need to argue that there are only finitely many linear isometries between two fixed linear subspace $U,W\subseteq V$.

According to Lemma \ref{polylem}(i), both of the closed unit balls of $U$ and $W$ are polytopes. Then Lemma \ref{polylem}(ii) says that both of them are the convex hull of their finitely many vertices, respectively. We know a linear isometry between $U$ and $W$ must send vertices to vertices, and by Lemma \ref{polylem}(ii) again the vertices of the closed unit ball span the linear space. Hence there are only finitely many linear isometries between $U$ and $W$, which guarantees that there are only finitely many rank-$r_0$ elements in this semigroup. We denote $R$ as this finite set.

Denote $L$ as the set of ``rank $<r_0$" elements of the generators $T_1,\dots,T_n$. Then every element in the semigroup can be expressed as a word in which the letters are elements in $L\cup R$, such that the partial composition of two adjacent letters in $R$ no longer lies in $R$. Let
$$
A=\{S_2\circ S_1\mid S_1,S_2\in R\}\setminus R\quad\textrm{and}\quad B=\{S_2\circ S_1\mid S_1\in L,S_2\in R\}.
$$
Then $A$ and $B$ are finite sets consist of elements of rank $<r_0$. By induction hypothesis, the semigroup $X$ generated by $L\cup A\cup B$ is finite. Moreover, we can see that every element of this semigroup lies in $X\cup R\cup\{S_2\circ S_1\mid S_1\in R,S_2\in X\}$. Hence this semigroup generated by $T_1,\dots,T_n$ is finite and we finish the proof by induction.
\end{proof}

\section{Proof of the main results}\label{Sec3}

We prove Proposition \ref{mainprop} in subsection \ref{sec3.1} and prove Theorem \ref{mainthm} in subsection \ref{sec3.2}.

\subsection{Proof of Proposition \ref{mainprop}}\label{sec3.1}

We proceed in the setting of Proposition \ref{mainprop}. We fix a big polarization $M=(B;\mathcal{M})$ of the arithmetic function field $K$ to use the machinery in subsection \ref{sec2.1} and the result in subsection \ref{expand}.

According to the hypothesis, the normalization of each curve $C_{-n}\subseteq\mathbb{A}_K^N$ is isomorphic to $\mathbb{G}_{m,K}$. Let $\pi_n\colon\mathbb{G}_{m,K}\to C_{-n}$ be a normalization map and let $P_n\colon\mathbb{G}_{m,K}\to\mathbb{A}_K^N$ be the composition of $\pi_n$ with the inclusion $C_{-n}\subseteq\mathbb{A}_K^N$. Notice that the we can choose $\pi_n$ and $P_n$ up to an automorphism of $\mathbb{G}_{m,K}$. We write down the expressions of $P_n$ and regard them as elements in $K[t,t^{-1}]^N$.

Since the maps $f|_{C_{-n}}\colon C_{-n}\to C_{-n+1}$ can be lifted to endomorphisms of the common normalization $\mathbb{G}_{m,K}$, for every $n\geq1$ we get $c_n\in K^{\times}$ and $m_n\in\mathbb{Z}\setminus\{0\}$ such that
$$
f(P_n(t))=P_{n-1}(c_nt^{m_n}).
$$
Adjusting those $P_{-n}$ by applying the automorphism $[-1]$ of $\mathbb{G}_{m,K}$ if necessary, we may assume that every $m_n$ is positive.

Our goal is to show that there exists $(\beta_n)_{n\geq0}\in(K^{\times})^{\mathbb{N}}$ such that
$$
\left\{P_n(t/\beta_n)\mid n\geq0\right\}
$$
is a finite set.

\begin{definition}\label{support}
Let $P\in K[t,t^{-1}]^N$. Write $P(t)=\sum\limits_{i\in\mathbb{Z}}a_it^i$ with $a_i\in K^N$. The \emph{support} of $P$ is defined as the set $\{i\in\mathbb{Z}\mid a_i\neq0\}$.
\end{definition}

\begin{remark}\label{supprmk}
Each map $P_n$ is birational onto its image. Hence for every $P_n$, the greatest common divisor of the elements of its (finite nonempty) support is $1$.
\end{remark}

We first show that there exists $(\lambda_n)_{n\geq0}\in(\overline{K}^{\times})^{\mathbb{N}}$ such that both the heights of the coefficients of $P_n(\lambda_nt)$ and the cardinality of the support of $P_n$ are uniformly bounded.

\begin{proposition}\label{bdhtsup}
Let $f$ be the regular endomorphism as in Proposition \ref{mainprop}. For every $n\geq0$, let $P_n\in K[t,t^{-1}]^N$ be the parametrization of $C_{-n}\subseteq\mathbb{A}_K^N$ as above. We have sequences $(a_n)_{n\geq0}\in(K^{\times})^{\mathbb{N}}$ and $(M_n)_{n\geq0}\in\mathbb{Z}_+^{\mathbb{N}}$ such that $f^n(P_n(t))=P_0(a_nt^{M_n})$ holds for every $n$. For each $n$, choose $\lambda_n\in\overline{K}^{\times}$ with $a_n\lambda_n^{M_n}=1$ and put $Q_n(t)=P_n(\lambda_nt)$. Then the following statements hold.
\begin{enumerate}
\item
There exists $H\in\mathbb{R}_{\geq0}$ such that every coefficient of every component of every $Q_n$ has height at most $H$.
\item
There exists $L\in\mathbb{Z}_+$ such that the cardinality of the support of every $P_n$ is at most $L$.
\end{enumerate}
\end{proposition}

\begin{proof}
For every $n$, we have
$$
f^n(Q_n(t))=P_0(t^{M_n}).
$$
Write $P_0(t)=(P_{0,1}(t),\dots,P_{0,N}(t))$ and write the $M_K$-constant $(c_v)_{v\in M_K}$ in Lemma \ref{explem} for $f$. We may enlarge $c_v$ such that
\begin{enumerate}
\item
for every $v\in M_K^{\infty}$, we have $c_v\geq\log\left(\sum\limits_{k=1}^{N}L_v(P_{0,k})\right)$ and
\item
for every $v\in M_{K}^0$, we have $c_v\geq\log\max\limits_{1\leq k\leq N}H_v(P_{0,k})$
\end{enumerate}
where we naturally extend Notation \ref{notation} to elements in $K[t,t^{-1}]$. Set
$$
H=\int_{M_K}c_vd\mu_K(v)\quad\text{and}\quad L=\lceil N\mathrm{e}^{\frac{2H}{\mu_K(M_K^{\infty})}}\rceil.
$$
We prove that the conclusions hold.

Write $Q_n(t)=(Q_{n,1}(t),\dots,Q_{n,N}(t))$ and write $Q_{n,k}(t)=\sum\limits_{j\in\mathbb{Z}}b_{k,j}^{(n)}t^j$. Denote $E_n=K(\lambda_n)$. Then all the coefficients $b_{k,j}^{(n)}$ lie in $E_n$. For every $w\in M_{E_n}$, we set $c_w=n_{w/v}c_v$ as in Remark \ref{MKconstrmk}(iii) and let $R_w=\mathrm{e}^{c_w}$.

For $w\in M_{E_n}^0$, we consider the Gauss absolute value $u$ on the function field $E_n(t)$. For our purpose, we let this absolute value be that characterized by $\|P\|_u=H_w(P)^{\frac{1}{n_{w/v}}}$ for every $P\in E_n[t]$. The exponent $\frac{1}{n_{w/v}}$ is used to make $u$ into an extension of $v\in M_K$. Applying Lemma \ref{explem} to $(E_n(t),u)$, we get $\|Q_n(t)\|_u\leq\mathrm{e}^{c_v}$. This means that all coefficients satisfy $\|b_{k,j}^{(n)}\|_w\leq R_w$.

For $w\in M_{E_n}^{\infty}$, we let $\sigma_w\colon E_n\hookrightarrow\mathbb{C}$ be the injection corresponding to $w$. Applying Lemma \ref{explem} to the standard absolute value on $\mathbb{C}$, we see that $\|\sigma_w(Q_n)(t)\|_{\infty}\leq R_w$ for every $t\in\mathbb{C}$ which lies on the unit circle. For every $1\leq k\leq N$, we have
$$
\sum\limits_{j\in\mathbb{Z}}\|b_{k,j}^{(n)}\|_w^2=\sum\limits_{j\in\mathbb{Z}}|\sigma_w(b_{k,j}^{(n)})|^2=\frac{1}{2\pi}\int_{0}^{2\pi}|\sigma_w(Q_{n,k})(\mathrm{e}^{i\theta})|^2d\theta\leq R_w^2.
$$
In particular, we also have $\|b_{k,j}^{(n)}\|_w\leq R_w$. Hence this inequality holds for every $w\in M_{E_n}$ and we get
$$
h(b_{k,j}^{(n)})=\frac{1}{[E_n:K]}\int_{M_{E_n}}\log^+\|b_{k,j}^{(n)}\|_wd\mu_{E_n}(w)\leq\frac{1}{[E_n:K]}\int_{M_{E_n}}c_wd\mu_{E_n}(w)=H.
$$
This finishes the verification of part (i).

To count the nonzero coefficients, we start with the inequality
\begin{align}\label{ineq1}
\sum\limits_{k=1}^N\sum\limits_{j\in\mathbb{Z}}\frac{\|b_{k,j}^{(n)}\|_w^2}{R_w^2}\leq N\quad(w\in M_{E_n}^{\infty}).
\end{align}
For a nonzero coefficient $b$, the product formula and its nonarchimedean bound imply
\begin{align*}
\int_{M_{E_n}^{\infty}}\log\frac{\|b\|_w}{R_w}d\mu_{E_n}(w)
&=-\int_{M_{E_n}^0}\log\|b\|_wd\mu_{E_n}(w)-\int_{M_{E_n}^{\infty}}\log R_wd\mu_{E_n}(w) \\
&\geq-\int_{M_{E_n}}c_wd\mu_{E_n}(w)=-[E_n:K]\cdot H.
\end{align*}
Denote $A=\mu_{E_n}(M_{E_n}^{\infty})=[E_n:K]\mu_K(M_K^{\infty})$. Using Jensen's inequality on the archimedean probability space $(M_{E_n}^{\infty},\mu_{E_n}/A)$, we get
\begin{align}\label{ineq2}
\frac{1}{A}\int_{M_{E_n}^{\infty}}\frac{\|b\|_w^2}{R_w^2}d\mu_{E_n}(w)\geq\exp\left(\frac{2}{A}\int_{M_{E_n}^{\infty}}\log\frac{\|b\|_w}{R_w}d\mu_{E_n}(w)\right)\geq\mathrm{e}^{\frac{-2H}{\mu_K(M_K^{\infty})}}.
\end{align}
Let $s_n$ be the total number of nonzero coefficients of $Q_n$. Integrating \eqref{ineq1} over $M_{E_n}^{\infty}$ while summing \eqref{ineq2} over those nonzero coefficients, we see that
$$
s_nA\cdot\mathrm{e}^{\frac{-2H}{\mu_K(M_K^{\infty})}}\leq NA.
$$
This gives the desired bound $s_n\leq L$ and finishes the proof of part (ii) because the supports of $P_n$ and $Q_n$ are same.
\end{proof}

We have shown that the support of each $P_n$ contains at most $L$ elements. Next, we investigate how the supports evolve. The concept of partial linear isometries (see Definition \ref{parlinisom}) will play a key role.

\begin{notation}\label{eP}
Let $P\in K[t,t^{-1}]^N$ and let $\{e_1,\dots,e_s\}\subseteq\mathbb{Z}$ be its support. Assume $1\leq s\leq L$, which shall always hold in the setting of our concern, and let $e_1<\cdots<e_s$. Denote $e(P)$ as the vector $(e_1,\dots,e_s,0,\dots,0)\in\mathbb{Z}^L$. In this form, we call the first $s$ places as ``active slots" and call the places of those additional zeros as ``padded slots".
\end{notation}

Recall we have $f(P_n(t))=P_{n-1}(c_nt^{m_n})$ for every $n\geq1$.

\begin{proposition}\label{supevolve}
Let $f$ be the regular endomorphism as above and let $d>1$ be its algebraic degree. Let $L\in\mathbb{Z}_+$ be the number given in Proposition \ref{bdhtsup}(ii). Then there is a finite collection $\mathcal{P}$ of partial linear isometries of $(\mathbb{R}^L,\|\cdot\|_{\infty})$ such that the following holds. For every $n\geq1$, there exists some map $T\colon S\to\mathbb{R}^L$ in $\mathcal{P}$ such that
$$
e(P_n)\in S\quad\text{and}\quad T(e(P_n))=\frac{m_n}{d}e(P_{n-1}).
$$
In particular, we have $m_n\|e(P_{n-1})\|_{\infty}=d\|e(P_n)\|_{\infty}$.
\end{proposition}

\begin{proof}
Let $\mathcal{I}=\{\alpha\in\mathbb{N}^L\mid\|\alpha\|_1\leq d\}$. Let $\mathcal{A}\subseteq M_L(\mathbb{R})=\mathrm{End}(\mathbb{R}^L)$ be the set of matrices that every row is an element of $\mathcal{I}$. Let $\mathcal{S}$ be the set of linear subspaces of $\mathbb{R}^L$ that can be defined by several equations of the form
$$
(\alpha-\beta)\cdot x=0\quad(x\in\mathbb{R}^L)
$$
in which $\alpha,\beta\in\mathcal{I}$. Since $\mathcal{I}$ is finite, so do $\mathcal{A}$ and $\mathcal{S}$. Denote
$$
\mathcal{P}=\left\{\frac{1}{d}A|_{S}\bigg|\ A\in\mathcal{A},\ S\in\mathcal{S},\ \frac{1}{d}A|_{S}\colon S\to\mathbb{R}^L\text{ is an isometry onto its image}\right\}.
$$
We shall show that this finite collection of partial linear isometries satisfies the requirement.

Write $e(P_n)=(e_1,\dots,e_L)$ and write $P_n(t)=\sum\limits_{j=1}^L a_jt^{e_j}$ for some $a_1,\dots,a_L\in K^N$ such that $a_j=0$ for the padded slots $j$. By expanding the expression formally, we find $(B_{\alpha})_{\alpha\in\mathcal{I}}\in(K^N)^{\mathcal{I}}$ such that
$$
f\left(\sum\limits_{j=1}^L a_jt^{k_j}\right)=\sum\limits_{\alpha\in\mathcal{I}}B_{\alpha}t^{\alpha\cdot k}
$$
holds for every $k=(k_1,\dots,k_L)\in\mathbb{Z}^L$.

Partition $\mathcal{I}=I_1\sqcup I_2\sqcup\cdots\sqcup I_t$ according to the value of $\alpha\cdot e(P_n)\ (\alpha\in\mathcal{I})$. Let $S\in\mathcal{S}$ be the linear subspace defined by the equations
$$
(\alpha-\beta)\cdot x=0,\quad\alpha,\beta\in I_j,\quad j=1,\dots,t\quad(x\in\mathbb{R}^L).
$$
Then by definition $e(P_n)\in S$.

We consider a matrix $A\in\mathcal{A}$ defined as follows. Let $l\in\{1,\dots,L\}$. If $l$ is an active slot of $e(P_{n-1})$, then there exists a unique $j_l\in\{1,\dots,t\}$ such that the $l$-th component of $e(P_{n-1})$ is equal to $\frac{\alpha\cdot e(P_n)}{m_n}$ for $\alpha\in I_{j_l}$. We arbitrarily choose an $\alpha\in I_{j_l}$ and let the $l$-th row of $A$ be $\alpha$. If $l$ is a padded slot of $e(P_{n-1})$, then we let the $l$-th row of $A$ be zero.

By definition we have $A\cdot e(P_n)=m_n\cdot e(P_{n-1})$. Therefore, it only remains to check that
$$
T=\frac{1}{d}A|_{S}\colon S\to\mathbb{R}^L
$$
is an isometry onto its image.

Let $\mathcal{H}$ be the finite set of codimension $1$ linear subspaces defined by equations
$$
(\alpha-\beta)\cdot x=0,\quad\alpha,\beta\in\mathcal{I}\text{ lie in different }I_j\quad(x\in\mathbb{R}^L).
$$
Let $U=S\setminus\bigcup\limits_{H\in\mathcal{H}}H$. We have $e(P_n)\in U$, hence $U$ is nonempty. Since the finitely many linear subspaces involved in here are all defined over $\mathbb{Q}$, it suffices to check
$$
\|Tx\|_{\infty}=\|x\|_{\infty}\quad\text{for}\quad x\in U\cap\mathbb{Q}^L.
$$
Pick such an $x=(x_1,\dots,x_L)$ and we may furthermore assume $x\in\mathbb{Z}^L$ by passing to an appropriate multiple. We need to verify $\|Ax\|_{\infty}=d\|x\|_{\infty}$.

By definition, for every $\alpha,\beta\in\mathcal{I}$, we have $(\alpha-\beta)\cdot x=0$ if and only if $(\alpha-\beta)\cdot e(P_n)=0$. Hence for a padded slot $j$ of $e(P_n)$ we have $x_j=0$, and the components of $x$ on the active slots of $e(P_n)$ are pairwise distinct. Write
$$
\tilde{P}(t)=\sum\limits_{j=1}^L a_jt^{x_j}\quad\text{and}\quad Q(t)=f(\tilde{P}(t))=\sum\limits_{\alpha\in\mathcal{I}}B_{\alpha}t^{\alpha\cdot x}.
$$
Then we have $\|e(\tilde{P})\|_{\infty}=\|x\|_{\infty}$. Since $f$ is a regular endomorphism, we also have $\|e(Q)\|_{\infty}=d\|e(\tilde{P})\|_{\infty}$.

When we divide $\mathcal{I}$ according to the value of $\alpha\cdot x\ (\alpha\in\mathcal{I})$, we still get the same partition $\mathcal{I}=I_1\sqcup I_2\sqcup\cdots\sqcup I_t$. This fact, together with the definition of the matrix $A$, lead to the equation $\|e(Q)\|_{\infty}=\|Ax\|_{\infty}$. Hence we get $\|Ax\|_{\infty}=d\|x\|_{\infty}$ and finish the proof.
\end{proof}

With these preparations, now we can prove Proposition \ref{mainprop}.

\proof[Proof of Proposition \ref{mainprop}]
We continue from the setting above. By considering the parametrizations of the normalizations of $C_{-n}$, we get a sequence $(P_n)_{n\geq0}\in K[t,t^{-1}]^N$ which satisfies
$$
f(P_n(t))=P_{n-1}(c_nt^{m_n})
$$
for every $n\geq1$. Proposition \ref{bdhtsup} says that the following statements hold.
\begin{enumerate}
\item
There exists a sequence $(\lambda_n)_{n\geq0}\in(\overline{K}^{\times})^{\mathbb{N}}$ and a nonnegative number $H$ such that every coefficient of every component of $P_n(\lambda_nt)$ has height at most $H$.
\item
There exists a positive integer $L$ such that the support of $P_n$ contains at most $L$ elements for every $n$.
\end{enumerate}
Then Proposition \ref{supevolve} says that there exists a finite collection $\mathcal{P}$ of partial linear isometries of $(\mathbb{R}^L,\|\cdot\|_{\infty})$ such that for every $n\geq1$, there is a map $T_n\colon S\to\mathbb{R}^L$ in $\mathcal{P}$ such that
$$
e(P_n)\in S\quad\text{and}\quad T_n(e(P_n))=\frac{m_n}{d}e(P_{n-1}).
$$
Put $r_n=\|e(P_n)\|_{\infty}>0$ and $x_n=e_n/r_n$. The isometry equality gives $m_nr_{n-1}=dr_n$, and hence we have $T_nx_n=x_{n-1}$ for every $n\geq1$.

We denote $T_n^{-1}$ as the total inverse partial linear isometry of $T_n$. Then $\{T_n^{-1}\mid n\geq0\}$ is also a finite set. Since the closed unit ball of $(\mathbb{R}^L,\|\cdot\|_{\infty})$ is a polytope, we can use Proposition \ref{linlem} to conclude that the semigroup $\mathcal{S}$ generated by $\{T_n^{-1}\mid n\geq0\}$ is also finite. For every $n\geq0$, we can see that $x_0$ lies in the domain of definition of $T=T_n^{-1}\circ\cdots\circ T_1^{-1}\in\mathcal{S}$ and $Tx_0=x_n$. Therefore, the finiteness of $\mathcal{S}$ implies that $\{x_n\mid n\geq0\}$ is finite. As we have mentioned in Remark \ref{supprmk}, the integral vectors $e(P_n)$ are primitive. Hence
$$
\{e(P_n)\mid n\geq0\}
$$
is also finite. In other words, the supports of $P_n$ have only finitely many possibilities.

Now we prove that there exists $(\beta_n)_{n\geq0}\in(K^{\times})^{\mathbb{N}}$ such that
$$
\left\{P_n(t/\beta_n)\mid n\geq0\right\}
$$
is a finite set.

Fix one of the finitely many possibilities of the support $\{e_1,\dots,e_s\}$. Since its greatest common divisor is $1$, we can find integers $u_1,\dots,u_s$ such that $\sum\limits_{j=1}^{s}u_je_j=1$. For a $P_n$ with this support we write $P_n(t)=(P_{n,1}(t),\dots,P_{n,N}(t))$ and
$$
P_{n,k}(t)=\sum\limits_{j=1}^{s}a_{k,j}^{(n)}t^{e_j},\quad k=1,\dots,N.
$$
Put $b_{k,j}^{(n)}=a_{k,j}^{(n)}\lambda_n^{e_j}$. Then the height of every $b_{k,j}^{(n)}\in\overline{K}$ is at most $H$. For each $j$ choose one nonzero element $a_j^{(n)}$ from the vector $(a_{1,j}^{(n)},\dots,a_{N,j}^{(n)})$, and let $b_j^{(n)}=a_j^{(n)}\lambda_n^{e_j}$. Define
$$
\beta_n=\prod\limits_{j=1}^{s}(a_{j}^{(n)})^{u_j}\in K^{\times},
$$
and let $B_n=\prod\limits_{j=1}^{s}(b_{j}^{(n)})^{u_j}=\beta_n\lambda_n$. Then the coefficients of $P_n(t/\beta_n)$ has the form $a_{k,j}^{(n)}\beta_n^{-e_j}=b_{k,j}^{(n)}B_n^{-e_j}$. Hence we have
$$
h(a_{k,j}^{(n)}\beta_n^{-e_j})\leq h(b_{k,j}^{(n)})+|e_j|h(B_n)\leq H+|e_j|H\sum\limits_{k=1}^{s}|u_k|.
$$
Since there are only finitely many possibilities of $\{e_1,\dots,e_s\}$, we see that the heights of the coefficients of $P_n(t/\beta_n)$ are uniformly bounded. Thus Northcott's finiteness property \ref{northcott} helps us conclude that $\left\{P_n(t/\beta_n)\mid n\geq0\right\}$ is a finite set.

Since changing the parametrization from $P_n(t)$ into $P_n(t/\beta_n)$ keeps the image being the curve $C_{-n}$, we finish the proof that there are only finitely many different curves in $(C_{-n})_{n\geq0}$.
\endproof

\subsection{Proof of Theorem \ref{mainthm}}\label{sec3.2}

In this subsection, we prove Theorem \ref{mainthm}. We first prove the first part, that is, the normalization $\widetilde{C}$ of the curve $C$ is isomorphic to either $\mathbb{A}^1$ or $\mathbb{G}_m$.

We find an arithmetic function field $K\subseteq\mathbb{C}$ such that the regular endomorphism $f$, the starting point $x$, and the curve $C\subseteq\mathbb{A}^N$ are defined over $K$. Since $\mathcal{O}_f(x)\cap C$ is infinite, we can use Theorem \ref{siegel} to conclude that the normalization $\widetilde{C}$ of $C$ is isomorphic to $\mathbb{P}_K^1\setminus D$ where $D$ is a reduced zero cycle on $\mathbb{P}_K^1$ with $\deg(D)\in\{1,2\}$ (see also the paragraph below Theorem \ref{siegel} for more informations). In our situation, the operation of taking normalization commutes with arbitrary base field extension because the curves are all integral and contain infinitely many rational points. If $\deg(D)=1$, then $D$ is a rational point and hence $\widetilde{C}$ is isomorphic to $\mathbb{A}^1$ (either over $K$ or over $\mathbb{C}$). If $\deg(D)=2$, then $\widetilde{C}$ is isomorphic to $\mathbb{G}_m$ over $\mathbb{C}$.

Now we prove that $C$ is $f$-periodic in the latter case. If $\deg(f)=1$, then $f$ is an automorphism and hence the DML conjecture holds for $f$ \cite[Theorem 1.3]{BGT10}. Thus we may assume that $\deg(f)>1$.

By passing $K$ to a finite extension if necessary, we may assume that $D$ consists of two rational points and hence $\widetilde{C}$ is isomorphic to $\mathbb{G}_m$ over $K$. Since $\mathcal{O}_f(x)\cap C$ is infinite, we can find a sequence of irreducible closed subcurves
$$
\cdots\stackrel{f}\rightarrow C_{-n}\stackrel{f}\rightarrow\cdots\stackrel{f}\rightarrow C_{-1}\stackrel{f}\rightarrow C_0=C
$$
of $\mathbb{A}_K^N$ such that each $C_{-n}$ has an infinite intersection with the orbit $\mathcal{O}_f(x)$.
Let $\overline{C_{-n}}$ be the projective closure of $C_{-n}$ in $\mathbb{P}_K^N$ and let $\widetilde{\overline{C_{-n}}}$ be its normalization. When regarding $f$ as an endomorphism of $\mathbb{P}_K^N$, we have $f(\overline{C_{-n}})=\overline{C_{-n+1}}$ for every $n\geq1$. As we have explained in the paragraph below Theorem \ref{siegel}, every $\widetilde{\overline{C_{-n}}}$ is isomorphic to $\mathbb{P}_K^1$. Also, every normalization $\widetilde{C_{-n}}$ admits an open immersion into $\widetilde{\overline{C_{-n}}}$ such that the complement is a reduced divisor of degree at most $2$. We claim that this complement consists of two rational points for every $n$. This is because the following diagrams commute,
\[
\xymatrixcolsep{5pc}\xymatrix{\widetilde{C_{-n}} \ar[r]^-{\widetilde{f|_{C_{-n}}}} \ar@{^{(}->}[d] & \widetilde{C_{-n+1}} \ar@{^{(}->}[d] \\
\widetilde{\overline{C_{-n}}} \ar[r]^-{\widetilde{f|_{\overline{C_{-n}}}}} & \widetilde{\overline{C_{-n+1}}}}
\]
the horizontal arrows at the bottom are surjective, and $\widetilde{\overline{C_0}}\setminus\widetilde{C_0}$ has been assumed to consist of two rational points.

As a result, every $\widetilde{C_{-n}}$ is isomorphic to $\mathbb{G}_{m,K}$. Then Proposition \ref{mainprop} says that $\{C_{-n}\mid n\geq0\}$ is a finite set. Hence $C_0=C$ is an $f$-periodic curve and we finish the proof.

\section*{Acknowledgements}
We are grateful to our advisor, Junyi Xie, for introducing the dynamical Mordell--Lang conjecture to us and his guidance through the past few years.

This work is supported by the National Natural Science Foundation of China Grant No. 12271007.

\bibliographystyle{alpha}
\bibliography{reference}

@book{BGT16,
    title={The Dynamical Mordell--Lang Conjecture},
    author={J. P. Bell and D. Ghioca and T. J. Tucker},
    publisher={American Mathematical Society},
    year={2016},
    series={Mathematics Surveys and Monographs},
    volume={\textbf{210}},
    address={Providence, RI}
}

@article{Xie17,
    title={The dynamical {M}ordell--{L}ang conjecture for polynomial endomorphisms of the affine plane},
    author={J. Xie},
    journal={Ast\'erisque},
    year={2017},
    volume={\textbf{394}},
    pages={vi+110}
}

@incollection{Xie,
    title={Around the dynamical {M}ordell--{L}ang conjecture},
    author={Junyi Xie},
    BOOKTITLE = {Algebraic, Complex, and Arithmetic Dynamics},
    SERIES = {Simons Symposia},
    PAGES = {59--98},
    PUBLISHER = {Springer Cham},
    YEAR = {2026}
}

@article{Bel06,
    title={A generalized {S}kolem--{M}ahler--{L}ech theorem for affine varieties},
    author={J. P. Bell},
    journal={J. London Math. Soc. (2)},
    year={2006},
    volume={\textbf{73}},
    number={2},
    pages={367--379}
}

@article{BGT10,
    title={The dynamical {M}ordell--{L}ang problem for \'etale maps},
    author={J. P. Bell and D. Ghioca and T. J. Tucker},
    journal={Amer. J. Math.},
    volume={\textbf{132}},
    number={6},
    year={2010},
    pages={1655--1675}
}

@article{Mor00,
    title={Arithmetic height functions over finitely generated fields},
    author={A. Moriwaki},
    journal={Invent. Math.},
    volume={\textbf{140}},
    number={1},
    pages={101--142},
    year={2000}
}

@article{Voj21,
    title={Roth's theorem over arithmetic function fields},
    author={P. Vojta},
    journal={Algebra $\&$ Number Theory},
    volume={\textbf{15}},
    number={8},
    pages={1943--2017},
    year={2021}
}

@article{GT09,
    title = {Periodic points, linearizing maps, and the dynamical {M}ordell--{L}ang problem},
    author = {D. Ghioca and T.J. Tucker},
    journal = {J. Number Theory},
    volume = {\textbf{129}},
    number = {6},
    pages = {1392--1403},
    year = {2009}
}

@article{Xie24,
    title={Algebraicity criteria, invariant subvarieties and transcendence problems from arithmetic dynamics},
    author={J. Xie},
    journal={Peking Math. J.},
    volume={\textbf{7}},
    number={1},
    pages={345--398},
    year={2024}
}

@unpublished{Zhong,
    title={Polynomial endomorphisms of $\mathbb{A}^2$ with many periodic curves},
    author={X. Zhong},
    note={arXiv:2508.13873v2}
}

@article{DFR,
    title={On the dynamical {M}anin--{M}umford conjecture for plane polynomial maps},
    author={R. Dujardin and C. Favre and M. Ruggiero},
    journal={J. Eur. Math. Soc. (JEMS)},
    year={2026},
    note={https://doi.org/10.4171/jems/1807}
}

@unpublished{JXZ,
    title={Cyclotomic integral points for affine dynamics},
    author={Z. Ji and J. Xie and G.-R. Zhang},
    note={arXiv:2511.13443v2}
}

@article{BJ00,
    title={Dynamics of regular polynomial endomorphisms of $\mathbb{C}^k$},
    author={E. Bedford and M. Jonsson},
    journal={Amer. J. Math.},
    volume={\textbf{122}},
    number={1},
    pages={153--212},
    year={2000}
}

@unpublished{YZ,
    title={Local height arguments toward the dynamical Mordell--Lang conjecture},
    author={S. Yang and A. Zheng},
    note={arXiv:2605.11676}
}

@misc{Stacks,
    title={The {S}tacks project},
    author={Authors Stacks},
    note={https://stacks.math.columbia.edu},
    year={2026}
}

@book{Rud91,
    title={Functional Analysis},
    author={W. Rudin},
    publisher={McGraw-Hill, Inc.},
    year={1991},
    series={Internat. Ser. Pure Appl. Math.},
    address={New York},
    note={2nd ed.}
}

\address{Beijing International Center for Mathematical Research, Peking University, Beijing 100871, China}

\email{ys-yx@pku.edu.cn}

~

\address{Beijing International Center for Mathematical Research, Peking University, Beijing 100871, China}

\email{zay@pku.edu.cn}

\end{spacing}
\end{document}